\documentclass[a4paper,12pt,reqno]{amsart}
\usepackage[utf8]{inputenc}
\usepackage[english]{babel}
\usepackage{amsmath, amsthm, amssymb, amsopn, amsfonts, amstext, enumerate, color, mathtools, mathrsfs, hyperref, enumitem, url, setspace}
\usepackage{cleveref, autonum}
\usepackage[initials,nobysame]{amsrefs}

\definecolor{OIblue}{RGB}{0,114,178}
\definecolor{OIorange}{RGB}{230,159,0}
\definecolor{OIgreen}{RGB}{0,158,115}
\definecolor{OIvermillion}{RGB}{213,94,0}
\definecolor{OIpurple}{RGB}{204,121,167}
\definecolor{OIyellow}{RGB}{240,228,66}
\definecolor{OIblack}{RGB}{0,0,0}

\hypersetup{
	colorlinks=true,
	linkcolor=OIblue,
	citecolor=OIgreen,
	urlcolor=black,
	linktoc=all
}

\usepackage[margin=0.9in]{geometry}
\numberwithin{equation}{section}

\usepackage{cleveref}

\newcommand{\C}{\mathcal{C}}

\newcommand{\K}{\mathcal{K}}

\newcommand{\N}{\mathbb{N}}

\newcommand{\R}{\mathbb{R}}

\newcommand{\Z}{\mathbb{Z}}

\newcommand{\loc}{{\rm loc}}

\newcommand{\dist}{{\mbox{\normalfont dist}}}

\newcommand{\Haus}{\mathcal{H}}

\newcommand{\norm}[1]{\left \| {#1} \right \| }
\newcommand{\seminorm}[1]{\left [ {#1} \right ] }

\DeclareMathOperator{\Tail}{Tail}

\DeclareMathOperator{\DG}{DG}
\DeclareMathOperator{\wDG}{wDG}

\def\XXint#1#2#3{{\setbox0=\hbox{$#1{#2#3}{\int}$ }
\vcenter{\hbox{$#2#3$ }}\kern-.6\wd0}}

\theoremstyle{plain}
\newtheorem{definition}{Definition}[section]
\newtheorem{theorem}[definition]{Theorem}

\newtheorem{lemma}[definition]{Lemma}

\theoremstyle{definition}
\newtheorem{remark}[definition]{Remark}

\renewcommand{\le}{\leqslant}
\renewcommand{\leq}{\leqslant}
\renewcommand{\ge}{\geqslant}
\renewcommand{\geq}{\geqslant}

\title[A fractional De Giorgi isoperimetric type inequality]{A fractional De Giorgi isoperimetric type inequality}

\author{Matteo Cozzi}
\address{M. Cozzi:
	Dipartimento di Matematica ``Federigo Enriques'',
	Universit\`a degli Studi di Milano,
	Via Saldini 50,
	20133 Milano,
	Italy}
\email{matteo.cozzi@unimi.it}

\author{Tomás Sanz-Perela}
\address{T. Sanz-Perela:
	Departament de Matem\`atiques i Inform\`atica,
	Universitat de Barcelona,
	Gran Via de les Corts Catalanes 585, 08007, Barcelona, Spain \quad and
	\newline
	Centre de Recerca Matem\`atica, Edifici C, Campus Bellaterra, 08193 Bellaterra, Spain}
\email{tomas.sanz.perela@ub.edu}



\begin{document}
	
\begin{abstract}
	We establish an isoperimetric type inequality for the level sets of functions in fractional Sobolev spaces. This answers a question posed by the first author in a previous paper. To obtain it, we work out a slight modification of some estimates for nonlocal interaction functionals established by Savin and Valdinoci. We also show how said isoperimetric inequality leads to the H\"older continuity of functions in (weak) fractional De Giorgi classes.
\end{abstract}

\maketitle

%\tableofcontents

\vspace{0.7cm}

\section{Introduction}	

\noindent
The relative isoperimetric inequality asserts that the perimeter of a set~$E$ in a fixed bounded domain~$\Omega$ controls the product between the volume of~$E$ and that of its complement~$\Omega \setminus E$. Restricting ourselves for simplicity to the case in which~$\Omega$ is a ball, this can be expressed more precisely via the inequality
\begin{equation}
\Big[ {|B_1 \cap E| |B_1 \setminus E|} \Big]^{\frac{n - 1}{n}} \le C_n P(E; B_1),
\end{equation}
where~$|\cdot|$ denotes the Lebesgue measure in~$\R^n$,~$P(\,\cdot\,; B_1)$ is the relative perimeter with respect to the unit ball~$B_1$, and~$C_n$ is a positive dimensional constant. See~\cite{G84}*{Chapter~1} or~\cite{M12}*{Chapter~12} for proofs of a slightly stronger version of this inequality.

A fundamental tool in De Giorgi's proof of the continuity of weak solutions to elliptic equations with rough coefficients is a functional form of the above isoperimetric inequality applied to level sets of Sobolev functions. This result, established in~\cite{DG57}, provides a quantitative expression of the fact that functions in~$W^{1,p}$ with~$p>1$ cannot exhibit jump discontinuities. The inequality can be stated as
\begin{equation}
	\label{Eq:LocalIsop}
	\Big[ |B_1 \cap \{ u \le h \}| | B_1 \cap \{ u \ge k \} | \Big]^{\frac{n - 1}{n}} \le \frac{C}{k - h} \, \| \nabla u \|_{L^p(B_1)} | B_1 \cap \{ h < u < k \} |^{\frac{p - 1}{p}}
\end{equation}
and holds for any two real numbers~$h < k$, any~$u \in W^{1, p}(B_1)$ with~$p \in (1, +\infty)$, and for some constant~$C > 0$ depending only on~$n$ and~$p$. In De Giorgi's regularity theory, this inequality (or some variation of it) is one of the tools used to obtain an oscillation decay estimate that eventually leads to the H\"older regularity of solutions---see, for instance,~\cites{V16,FernRealRosOton-PDEbook,I26} for more detailed discussions.

\medskip

In the early 2000s, the study of integrodifferential equations gained a lot of attention within the PDE community---even though nonlocal operators were already studied in probability or harmonic analysis. This development was driven both by applications---such as anomalous diffusion, phase transitions, finance, and image processing---and by the realization that many classical techniques could be adapted, at least in part, to the nonlocal setting.

The fractional Laplacian~$(-\Delta)^s$, for~$s \in (0, 1)$, is the paradigmatic example of nonlocal operator affected by this growing interest. In the early stages of the theory, a decisive role was played by the so-called \emph{extension problem}, a method that allows one to write a nonlocal equation as a local one over a space with one additional dimension and which thus grants access to powerful tools from the local PDE theory---see~\cite{CS07}.
Soon, however, the focus expanded to more general classes of nonlocal operators that do not admit local extension problems. This shift led to the development of genuinely nonlocal techniques, designed to handle directly the long-range interactions encoded in the integrodifferential structure of the equations.
For more details we refer the reader to,~e.g.,~\cites{FracPDEsBook, RecentDevelopmentsNonlocal}.

In this setting, a natural question was whether the regularity theory of De Giorgi could be adapted to the fractional framework.
This was first addressed in~\cite{CaffarelliVasseur} with the help of the local extension problem and in~\cite{CCV11} through the development of intrinsically nonlocal techniques. Afterwards, several other works expanded on these ideas, including~\cite{Cozzi-DeGiorgi} by the first author, where a notion of~\emph{fractional De Giorgi class} was first explicitly introduced and systematically studied. Without entering now into much detail (see the forthcoming Subsection~\ref{Subsec:applintro}), these are subsets of the fractional Sobolev space~$W^{s,p}$ made up of functions that admit a suitable Caccioppoli type inequality---a feature enjoyed by solutions of nonlocal problems, which is understood in~\cites{CCV11,Cozzi-DeGiorgi} as the starting point to establish their regularity. To carry forward this objective, one needs a replacement for the local inequality~\eqref{Eq:LocalIsop}. In~\cites{CCV11,Cozzi-DeGiorgi}, this came in the form of an especially strong version of the fractional Caccioppoli inequality, which is satisfied by solutions of nonlocal PDEs and contains a genuinely nonlocal extra term on its left-hand side, not present in its local counterparts. Notwithstanding, in order to obtain regularity estimates stable with respect to the parameter~$s$, in~\cite{Cozzi-DeGiorgi} a partial analogue of~\eqref{Eq:LocalIsop} was obtained for functions lying in the fractional Sobolev space~$W^{s, p}$ for~$s$ close to~$1$. Therefore, the following natural question (later made explicit in~\cite{C19}) remained open:
\vspace{2pt}
$$
\mbox{Does a version of~\eqref{Eq:LocalIsop} hold true in the space~$W^{s, p}$ for all~$p \in (1, +\infty)$ and~$s \in \left[ \frac{1}{p}, 1 \right)$?}
\vspace{2pt}
$$
Notice that this is not possible for~$s \in \left( 0, 1/p \right)$, since, in this regime, characteristic functions of smooth sets belong to~$W^{s, p}$---see the forthcoming Remark~\ref{Remark:Characteristic}.

In this paper, we give an affirmative answer to the above question.

\subsection{A fractional isoperimetric De Giorgi inequality}

In order to state our main result, we need to fix some notation.
As is customary, given~$s \in (0,1)$,~$p \in [1, +\infty)$, and an open set~$\Omega \subset \R^n$, we define the fractional Sobolev space~$W^{s,p}(\Omega)$ as
\begin{equation}
	W^{s,p}(\Omega)
	\coloneqq
	\Big\{
	{u \in L^p(\Omega)
	\; : \;
	[u]_{W^{s,p}(\Omega)} < +\infty}
	\Big\}.
\end{equation}
where~$[u]_{W^{s,p}(\Omega)}$ is the so-called Gagliardo seminorm, given by
\begin{equation} \label{Wspsemi}
	[u]_{W^{s,p}(\Omega)}
	\coloneqq
	\left( \int_{\Omega}
	\int_{\Omega}
	\frac{|u(x)-u(y)|^p}{|x-y|^{n+sp}}
	\, dx dy \right)^{\frac{1}{p}}.
\end{equation}
This seminorm, when multiplied by an apropriate normalizing constant depending on~$n$,~$s$, and~$p$, converges as~$s\uparrow1$ to the classical integer-order Sobolev seminorm~$\|\nabla u\|_{L^p(\Omega)}$. 
For a gentle introduction to this topic we refer the interested reader to~\cite{HGuide}. Finally, we introduce the following auxiliary function, that we shall use to quantify the dependence of our isoperimetric inequality on the Lebesgue measure of the \emph{interface}~$B_1 \cap \{ h < u < k\}$. For a given parameter~$\alpha \ge 1$, we consider the strictly increasing function~$\Psi_\alpha: [0, 1] \to [0, 1]$ defined for~$t \in (0, 1]$ by
\begin{equation} \label{Eq:Psialphadef}
	\Psi_\alpha(t) \coloneqq \begin{dcases}
	t^{\alpha - 1} & \quad \mbox{if } \alpha > 1, \\
	\frac{1}{\left| \log \! \left(\frac{t}{e} \right) \right|} & \quad \mbox{if } \alpha = 1,
	\end{dcases}
\end{equation}
and continuously extended to the whole~$[0, 1]$ by setting~$\Psi_\alpha(0) \coloneqq 0$.

Having this defined, we can now state our fractional isoperimetric inequality.

\begin{theorem}\label{Th:DeGiorgiIsop}
	Let~$n \in \N$,~$s \in (0, 1)$, and~$p \in (1, +\infty)$ be such that~$sp \geq 1$. Then, there exist two constants~$C \ge 1$ and~$\beta > 0$, depending only on~$n$,~$s$, and~$p$, such that
	\begin{equation}
		\Big[ {|B_1 \cap \{ u \le h \}||B_1 \cap \{ u \ge k \}|} \Big]^{\! \beta} \le \frac{C}{k - h} \, [u]_{W^{s, p}(B_1)} \Psi_{sp} \! \left( \frac{|B_1 \cap \{ h < u < k \}|}{|B_1|} \right)^{\!\frac{1}{p}},
	\end{equation}
	for every~$u \in W^{s, p}(B_1)$ and every two real numbers~$h < k$.
\end{theorem}

When~$sp > 1$, the above inequality has a power like structure which is completely analogous to the classical~\eqref{Eq:LocalIsop}. As~$sp$ gets closer and closer to~$1$, such inequality however deteriorates, since the power involved in the definition of~$\Psi_\alpha$ for~$\alpha > 1$ approaches zero. Nevertheless, when~$s p = 1$ a weaker form of inequality still holds true, involving on its right-hand side a logarithmic dependence on the measure of the interface.

As illustrated by this and previous discussions,~$sp = 1$ is a limiting case. Indeed, when~$sp < 1$ this kind of inequalities cannot hold, as shown in the following remark.

\begin{remark}
	\label{Remark:Characteristic}
	As we have already mentioned, a characteristic function of a set can belong to~$W^{s,p}$ when~$sp<1$. To see this, consider the cylinders
	\begin{equation}
		E \coloneqq B_1' \times (-1, 1), \quad E^+ \coloneqq B_1' \times (0, 1), \quad \text{where } B_1' \coloneqq \Big\{ {x' \in \R^{n-1} \, : \, |x'|<1} \Big\},
	\end{equation}
	as well as the characteristic function~$u \coloneqq \chi_{E^+}$. Then, we have
	\begin{equation}
		[u]_{W^{s,p}(E)}^p = 2 \int_0^1 \int_{-1}^0 \Bigg( {\int_{B_1'} \int_{B_1'} \big( {|x'-y'|^2 + |x_n-y_n|^2} \big)^{-\frac{n+sp}{2}} dx' dy'} \Bigg) \, dx_n dy_n,
	\end{equation}
	integral which, through the changes of coordinates~$x' = |x_n - y_n| w'$,~$y' = |x_n - y_n|(w' + z')$ and~$x_n = t$,~$y_n = - \tau$, can be estimated as
	\begin{equation}
		[u]_{W^{s,p}(E)}^p \le C_{n, s, p} \int_0^1 \int_0^1 \frac{dt d\tau}{(t + \tau)^{1 + sp}}, \quad \mbox{with } C_{n, s, p} \coloneqq 2 |{B_1'}| \int_{\R^n} \big( {1 + |z'|^2} \big)^{- \frac{n+sp}{2}} dz'.
	\end{equation}
	Note that the double integral in the variables~$t$ and~$\tau$ is finite if and only if~$sp<1$.	Hence, $u\in W^{s,p}(E)$ for~$sp<1$. However, the interface~$\{0<u<1\}$ has Lebesgue measure zero, so it cannot control the size of sub/super-level sets of the function~$u$, as in the  isoperimetric type inequalities of Theorem~\ref{Th:DeGiorgiIsop}.
\end{remark}

The proof of Theorem~\ref{Th:DeGiorgiIsop} relies on an accurate estimate for a nonlocal interaction functional applied to the sub- and superlevel sets of the function~$u$. More precisely, after a simple rescaling it can be assumed that~$h=0$ and~$k = 1$. Then, recalling~\eqref{Wspsemi}, one observes that
\begin{equation} \label{Eq:seminormcontrolsinteraction}
	[u]_{W^{s, p}(B_1)}^p \ge \int_A \int_B \frac{dx dy}{|x - y|^{n + s p}} \eqqcolon I(A,B),
\end{equation}
where~$A \coloneqq B_1 \cap \{u\leq0\}$ and~$B \coloneqq B_1 \cap \{u\geq 1\}$. As the computations in Remark~\ref{Remark:Characteristic} suggest, when~$sp \ge 1$ the nonlocal interaction~$I(A, B)$ gets large as the measure of the interface $E \coloneqq B_1 \cap \{ 0<u<1\}$ becomes small. To complete the proof of the isoperimetric inequalities, we make this observation quantitative and obtain a lower bound for~$I(A, B)$ in terms of the sizes of~$A$,~$B$, and~$E$. We achieve this by revisiting the proof of~\cite{SV14}*{Proposition~4.3} by Savin and Valdinoci, keeping track of how the regulating constants of their inequality depend on some relevant quantities. The result of this analysis is the following reworked inequality.

\begin{lemma} \label{Lemma:SVest}
	Let~$n \in \N$ and~$\alpha \in [1, +\infty)$. Let~$A, B \subset Q_1 \coloneqq (0,1)^n$ be disjoint measurable sets satisfying
	\begin{equation} \label{ABaresigmafat}
		\min \big\{ {|A|, |B|} \big\} \ge \sigma,
	\end{equation}
	for some~$\sigma \in (0, 1)$. Set~$E \coloneqq Q_1 \setminus \left( A \cup B \right)$. Then, there exist two constants~$c_\sharp \in (0, 1)$ and~$q \ge 1$, depending only on~$n$ and~$\alpha$, such that, if~$|E| \in (0, c_\sharp \, \sigma^q]$, then
	\begin{equation} \label{Eq:interacest}
		\int_A \int_B \frac{dx dy}{|x - y|^{n + \alpha}} \ge c_\sharp \, \frac{\sigma^q}{\Psi_\alpha \big( {|E|} \big)}.
	\end{equation}
\end{lemma}

We refer to~\Cref{Sec:MainEst} for the proof of this result. After Lemma~\ref{Lemma:SVest} is obtained, Theorem~\ref{Th:DeGiorgiIsop} is readily deduced from it---see~\Cref{Sec:MainResult} for the details.

\subsection{An application to weak fractional De Giorgi classes} \label{Subsec:applintro}

Besides their intrinsic interests, isoperimetric inequalities like~\eqref{Eq:LocalIsop} can be used to establish the regularity of solutions to PDEs, minimizers of elliptic functionals, and, more generally, elements of De Giorgi classes. In~\cite{Cozzi-DeGiorgi}, a fractional version of these classes was introduced, that we now proceed to describe following the more stripped-down approach of~\cite{C19}.

Let~$n \in \N$,~$p \in (1, +\infty)$, and~$s \in (0, 1)$. Given a measurable function~$u: \R^n \to \R$, we consider the quantity
$$
\Tail_{s,p} \! \big( {u; B_R(x_0)} \big) \coloneqq \left( \int_{\R^n \setminus B_R(x_0)} \frac{|u(x)|^{p - 1}}{|x - x_0|^{n + s p}} \, dx \right)^{\! \frac{1}{p - 1}},
$$
for~$x_0 \in \R^n$ and~$R > 0$, which is often called~\emph{nonlocal tail} of~$u$---notice the slight change of notation with respect to~\cite{C19}. We say that~$u \in L^{p - 1}_s(\R^n)$ if~$\Tail_{s, p} \! \big( {u; B_R(x_0)} \big) < +\infty$ for every~$x_0 \in \R^n$ and~$R > 0$---or, equivalently, if~$u \in L^{p - 1}_\loc(\R^n)$ and~$\Tail_{s, p} \! \big( {u; B_R(x_0)} \big) < +\infty$ for some~$x_0 \in \R^n$ and~$R > 0$. We also indicate its sub- and superlevel sets inside the ball~$B_R(x_0)$ as
$$
A^-\big( {k, B_R(x_0)} \big) \coloneqq \{ u < k \} \cap B_R(x_0) \quad \mbox{and} \quad A^+ \big( {k, B_R(x_0)} \big) \coloneqq \{ u > k \} \cap B_R(x_0),
$$
for any~$k \in \R$. Thanks to this terminology, we can now define fractional De Giorgi classes. In what follows,~$\Omega$ always denotes a bounded and open subset of~$\R^n$, while to indicate the positive and negative parts of~$u$ we write, as usual,~$u_+ \coloneqq \max \{ u, 0 \}$ and~$u_- \coloneqq \max \{ -u, 0 \}$.

\begin{definition}
Let~$d, H, \lambda$ be three non-negative real numbers. A function~$u \in L^{p - 1}_s(\R^n)$ is said to belong to~$\DG_{\pm}^{s, p}(\Omega; d, H, \lambda)$ if~$u \in W^{s, p}(\Omega)$ and it satisfies
\begin{equation} \label{Eq:DeGclass}
\begin{split}
& [(u - k)_\pm]_{W^{s, p}(B_r(x_0))}^p + \int_{B_r(x_0)} \big( {u(x) - k} \big)_{\pm} \bigg( {\int_{\R^n} \frac{\big( {u(y) - k} \big)_\mp^{p - 1}}{|x - y|^{n + s p}} \, dy} \bigg) dx \\
& \hspace{35pt} \le H \, \Bigg\{ {R^\lambda d^p \left| {A^\pm \big( {k, B_R(x_0)} \big)} \right| + \frac{R^{(1 - s)p}}{(R - r)^p} \| (u - k)_\pm \|_{L^p(B_R(x_0))}^p} \Bigg. \\
& \hspace{35pt} \quad\,\, \Bigg. {+ \frac{R^{n + s p}}{(R - r)^{n + sp}} \, \| (u - k)_\pm \|_{L^1(B_R(x_0))} \Tail_{s,p} \! \big( {(u - k)_\pm; B_r(x_0)} \big)^{p - 1}} \Bigg\},
\end{split}
\end{equation}
for every point~$x_0 \in \Omega$, radii~$0 < r < R < \dist(x_0, \partial \Omega)$, and level~$k \in \R$. We then set~$\DG^{s, p}(\Omega; d, H, \lambda) \coloneqq \DG_-^{s, p}(\Omega; d, H, \lambda) \cap \DG_+^{s, p}(\Omega; d, H, \lambda)$.
\end{definition}

In~\cite{Cozzi-DeGiorgi}, it has been proved that solutions to elliptic equations driven by integrodifferential operators like the fractional~$p$-Laplacian and minimizers of (possibly non-differentiable) nonlocal energies involving Gagliardo type seminorms belong to these classes. Furthermore, the elements of such classes were shown there to be locally bounded and H\"older continuous. It is remarkable that the second term on the first line of~\eqref{Eq:DeGclass} is used crucially to establish their continuity. Note that this term is purely nonlocal and has no analogue in standard, integer-order De Giorgi classes---as considered, for instance, in~\cite{G03}. It is therefore natural to ask what happens when one removes such term from~\eqref{Eq:DeGclass}. This brings us to the following definition of \emph{weak} fractional De Giorgi classes, originally appeared in~\cite{C19}.

\begin{definition}
Let~$d, H, \lambda$ be three non-negative real numbers. A function~$u \in L^{p - 1}_s(\R^n)$ is said to belong to~$\wDG_{\pm}^{s, p}(\Omega; d, H, \lambda)$ if~$u \in W^{s, p}(\Omega)$ and it satisfies
\begin{equation} \label{Eq:weakDeGclass}
\begin{split}
& [(u - k)_\pm]_{W^{s, p}(B_r(x_0))}^p \phantom{+ \int_{B_r(x_0)} \big( {u(x) - k} \big)_{\pm} \bigg( {\int_{\R^n} \frac{\big( {u(y) - k} \big)_\mp^{p - 1}}{|x - y|^{n + s p}} \, dy} \bigg) dx} \\
& \hspace{35pt} \le H \, \Bigg\{ {R^\lambda d^p \left| {A^\pm \big( {k, B_R(x_0)} \big)} \right| + \frac{R^{(1 - s)p}}{(R - r)^p} \| (u - k)_\pm \|_{L^p(B_R(x_0))}^p} \Bigg. \\
& \hspace{35pt} \quad\,\, \Bigg. {+ \frac{R^{n + s p}}{(R - r)^{n + sp}} \, \| (u - k)_\pm \|_{L^1(B_R(x_0))} \Tail_{s,p} \! \big( {(u - k)_\pm; B_r(x_0)} \big)^{p - 1}} \Bigg\},
\end{split}
\end{equation}
for every point~$x_0 \in \Omega$, radii~$0 < r < R < \dist(x_0, \partial \Omega)$, and level~$k \in \R$. We then set~$\wDG^{s, p}(\Omega; d, H, \lambda) \coloneqq \wDG_-^{s, p}(\Omega; d, H, \lambda) \cap \wDG_+^{s, p}(\Omega; d, H, \lambda)$.
\end{definition}

We point out that a different notation was adopted in~\cite{C19} for these classes:~$\wDG^{s, p}$ here corresponds to~$\widetilde{\DG}^{s, p}$ there.

As shown in~\cite{C19}, the weaker Caccioppoli type inequality~\eqref{Eq:weakDeGclass} still leads to~$L^\infty$ estimates. As a result, the elements of these larger De Giorgi classes are locally bounded functions. Their continuity, however, has been ruled out in~\cite{C19} for the case~$s p < 1$ (and~$n = 1$), by means of an explicit counterexample---essentially, the same characteristic function of Remark~\ref{Remark:Characteristic}. The question of whether this was not the case when~$s p \ge 1$ was left open in~\cite{C19}. We answer it here through the following result.

\begin{theorem} \label{Th:DeGHolder}
Let~$n \in \N$,~$s \in (0, 1)$, and~$p \in (1, +\infty)$ be such that~$s p \ge 1$. Let~$\Omega \subset \R^n$ be a bounded open set and~$u \in \wDG^{s, p}(\Omega; d, H, \lambda)$ for some~$d, H, \lambda \ge 0$. Then,~$u \in C^\alpha_\loc(\Omega)$ and it holds
$$
\| u \|_{C^\alpha(B_R(x_0))} \le \frac{C}{R^\alpha} \bigg( {R^{-\frac{n}{p}} \| u \|_{L^p(B_{2 R}(x_0))} + R^{\frac{sp}{p - 1}} \Tail_{s, p} \! \big( {u; B_R(x_0)} \big) + R^{\frac{\lambda + s p}{p}} d} \bigg),
$$
for every~$x_0 \in \Omega$,~$R \in \big( {0,  \dist(x_0, \partial \Omega)/2} \big)$, and for some constants~$\alpha \in (0, 1)$ and~$C \ge 1$ depending only on~$n$,~$s$,~$p$,~$H$, and~$\lambda$.
\end{theorem}

Theorem~\ref{Th:DeGHolder} follows from the arguments of~\cites{Cozzi-DeGiorgi,C19}, once one establishes a suitable ``growth lemma'' for functions in the weak fractional De Giorgi classes. This lemma is obtained by replicating the argument presented in~\cite{Cozzi-DeGiorgi}*{Lemma~6.3} for the case of~$s$ close to~$1$, in light of the newly found isoperimetric inequalities of Theorem~\ref{Th:DeGiorgiIsop}. We provide the details in~\Cref{Sec:Application}.

We conclude the introduction by pointing out that inequalities like~\eqref{Eq:weakDeGclass} may also appear when dealing with local, integer-order equations, as in~\cite{M11}. Consequently, Theorem~\ref{Th:DeGHolder} could be of interest for such applications as well.

\medskip

\subsection*{Acknowledgments}
Both authors have been supported by the grants PID2021-123903NB-I00 (Spain$^*$),~RED2022-134784-T (Spain$^*$), and by the Madrid Government (Comunidad de Madrid – Spain) under the multiannual Agreement with UAM in the line for the Excellence of the University Research Staff in the context of the V M. PRICIT (Regional Programme of Research and Technological Innovation). The first author has also been supported by the GNAMPA-INdAM project ``Equazioni nonlocali e nonlineari: alcune questioni di esistenza, rigidit\`a e regolarit\`a'' CUP E53C25002010001 (Italy) and by the PRIN project 20229M52AS\_004 ``Partial differential equations and related geometric-functional inequalities'' (Italy), while the second author by the grants~PID2020-113596GB-I00 (Spain$^*$) and~PID2024-156429NB-I00 (Spain$^*$).\\
$^*\,$These projects have been funded by MCIN/AEI/10.13039/501100011033 and by ``ERDF A way of making Europe”.
	
\section{Proof of Lemma~\ref{Lemma:SVest}}		
\label{Sec:MainEst}

\noindent
We revisit the proof of~\cite{SV14}*{Proposition~4.3}, tracking down the dependence on~$\sigma$ of the constant labeled there as~$\delta$. 
For brevity, we shall write~$\varepsilon \coloneqq |E|$.

The core of the proof consists in proving the following main estimate:
given~$\gamma \in (0, 1)$ and setting~$\Gamma \coloneqq 1 / \gamma$, we claim that there exists a constant~$c_\star \in \left( 0, 1/10 \right]$, depending only on~$n$ and~$\alpha$, such that if~$\gamma \le c_\star \sigma^{\frac{2 n - 1}{n}}$ and~$\varepsilon \leq \gamma \sigma / 4 $, then
\begin{equation} 
	\label{Eq:mainintest}
\iint_{\big\{ {(x, y) \in A \times B \, : \, \gamma r \le |x - y| \le \Gamma r} \big\}} \! \frac{dx dy}{|x - y|^{n + \alpha}} \ge c_\star \sigma^{\frac{2n - 1}{n}} r^{1 - \alpha} \quad \mbox{for all } r \in \left [ \dfrac{\varepsilon}{\gamma}, \dfrac{\sigma}{4} \right] \mbox{ s.t.~} \frac{1}{r} \in \N.
\end{equation}
Once this is established in the upcoming Step~1, the result will follow by using~\eqref{Eq:mainintest} with values of~$r$ comparable to~$\varepsilon$---see Step~2 for the details.

\medskip

\noindent
\textbf{Step 1.  Proof of the main estimate.}
To establish~\eqref{Eq:mainintest}, we decompose~$Q_1$ (up to sets of measure zero) into a family~$\K$ of~$r^{-n}$ non-overlapping open cubes of sides~$r$, with $r$ as in~\eqref{Eq:mainintest}, that is
$$
\K \coloneqq \Big\{ {(0, r)^n + r v : v \in \{ 0, 1, \ldots, r^{-1} - 1 \}^n} \Big\}.
$$
We then define
\begin{equation}
\begin{split}
\K_E & \coloneqq \left\{ Q \in \K : |Q \cap E| \ge \frac{1}{3} \, |Q| \right\}, \\
\K_B & \coloneqq \left\{ Q \in \K \setminus \K_E : |Q \cap B| \ge \frac{1}{3} \, |Q| \right\}, \\
\K_A & \coloneqq \vphantom{\bigg\{} \K \setminus \Big( {\K_E \cup \K_B} \Big), \\
Q_E & \coloneqq \vphantom{\bigg\{} \bigcup_{Q \in \K_E} Q, \quad Q_B \coloneqq \bigcup_{Q \in \K_B} Q, \quad Q_A \coloneqq \bigcup_{Q \in \K_A} Q.
\end{split}
\end{equation}
Observe that the sets~$\K_A$,~$\K_B$, and~$\K_E$ form a partition of~$\K$. Also,
\begin{equation} \label{massforKA}
|Q \cap A| \ge \frac{1}{3} \, |Q| \quad \mbox{for all } Q \in \K_A.
\end{equation}
Thanks to the definition of~$\K_E$ and~\eqref{massforKA}, we infer that
\begin{equation} \label{QEest}
|Q_E| = \sum_{Q \in \K_E} |Q| \le 3 \sum_{Q \in \K_E} |Q \cap E| \le 3 |E| = 3 \varepsilon
\end{equation}
and
\begin{equation} \label{QAest}
|Q_A| = \sum_{Q \in \K_A} |Q| \le 3 \sum_{Q \in \K_A} |Q \cap A| \le 3 |A|.
\end{equation}

We consider three different cases depending on the sizes of the sets $|Q_B \cap A|$ and $|Q_A \cap B|$.

\smallskip
\noindent
$\bullet$ \textbf{Case 1.}
Suppose that
\begin{equation} \label{QBinAbig}
|Q_B \cap A| \ge r.
\end{equation}
First, note that~$Q \subset B_{\Gamma r}(x)$ for all~$Q \in \K$ and~$x \in Q$, provided~$\gamma < 1/\sqrt{n}$. 
Hence, under this assumption on $\gamma$, given~$Q \in \K_B$ we estimate
\begin{equation}
\begin{split}
\int_{Q \cap A} \left( \int_{B \cap \left( B_{\Gamma r}(x) \setminus B_{\gamma r}(x) \right)} \frac{dy}{|x - y|^{n + \alpha}} \right) dx & \ge \int_{Q \cap A} \left( \int_{(B \cap Q) \setminus B_{\gamma r}(x)} \frac{dy}{|x - y|^{n + \alpha}} \right) dx \\
& \ge \int_{Q \cap A} \frac{\left| (B \cap Q) \setminus B_{\gamma r}(x) \right|}{(\sqrt{n} r)^{n + \alpha}} \, dx \\
& \ge \frac{|Q \cap A| \left( \frac{1}{3} - |B_1| \gamma^n \right)}{n^{\frac{n + \alpha}{2}} r^\alpha} \ge c_1 |Q \cap A| \, r^{-\alpha},
\end{split}
\end{equation}
for some~$c_1 > 0$ depending only on~$n$ and~$\alpha$, provided~$\gamma\le (6 |B_1|)^{-\frac{1}{n}}$. 
Thanks to this and~\eqref{QBinAbig}, we obtain that
\begin{equation}
\begin{split}
\iint_{\big\{ {(x, y) \in A \times B \, : \, \gamma r \le |x - y| \le \Gamma r} \big\}} \frac{dx dy}{|x - y|^{n + \alpha}} & \ge \sum_{Q \in \K_B} \int_{Q \cap A} \left( \int_{B \cap \left( B_{\Gamma r}(x) \setminus B_{\gamma r}(x) \right)} \frac{dy}{|x - y|^{n + \alpha}} \right) dx \\
& \ge c_1 \, r^{-\alpha} \sum_{Q \in \K_B} |Q \cap A| =  c_1 \, r^{-\alpha} |Q_B \cap A| \\
& \ge c_1 \, r^{1 - \alpha},
\end{split}
\end{equation}
which gives~\eqref{Eq:mainintest} provided $\gamma < \min \Big\{ {1/\sqrt{n}, (6 |B_1|)^{-\frac{1}{n}}} \Big\}$. 

\smallskip

\noindent
$\bullet$ \textbf{Case 2.} If~$|Q_A \cap B| \ge r$, then the same conclusion is reached by swapping the roles of~$A$ and~$B$ in the previous argument.

\smallskip

\noindent
$\bullet$ \textbf{Case 3.} It holds
\begin{equation}
	\label{Eq:QBinA/QAinBsmall}
	\max \Big\{ {|Q_B \cap A|, |Q_A \cap B|} \Big\} < r.
\end{equation}
We further divide our analysis depending on whether~$n = 1$ or~$n \ge 2$.

\smallskip

\noindent
\textbf{- Case 3.1. Dimension~$n=1$.}
We first notice that if~$\varepsilon \leq \gamma r$ with~$\gamma < 1/3$, then $\K_E = \varnothing$.
Indeed, if~$\K_E \neq \varnothing$ then there exists~$Q_\star \in \K_E$ and thus
\begin{equation}
	\frac{r}{3} = \frac{|Q_\star|}{3} \leq |Q_\star \cap E| \leq |E| = \varepsilon \leq \gamma r,
\end{equation}
which is a contradiction for~$\gamma < 1/3$.
Thus, we have that~$Q_1 = (0,1)$ coincides, up to sets of zero measure, with~$Q_A \cup Q_B$.
Hence, using~\eqref{ABaresigmafat} and~\eqref{Eq:QBinA/QAinBsmall}, we obtain
$$
|Q_A| \ge |Q_A \cap A| = |A|  - |Q_B \cap A| \ge \sigma - r  \ge \frac{\sigma}{2}
$$
since~$r\leq \sigma/2$, showing that~$\K_A \neq \varnothing$.
Similarly,~$|Q_B|\geq \sigma/2$ and thus~$\K_B \neq \varnothing$. 

From the previous considerations, it is not difficult to show that there exist two intervals~$Q^{(a)} \in \K_A$ and~$Q^{(b)} \in \K_B$ such that~$\dist \left( Q^{(a)}, Q^{(b)} \right) \in \left\{ r, 2 r \right\}$.
Moreover, for every~$x \in Q^{(a)}$ and~$y \in Q^{(b)}$ we have that~$r < |x - y| < 4 r \le \Gamma r$, provided~$\gamma \le 1/4$. Consequently,~$Q^{(b)} \subset B_{\Gamma r}(x) \setminus B_{\gamma r}(x)$ for all~$x \in Q^{(a)}$ and thus
\begin{equation}
	\begin{split}
	\int_{A} \left( \int_{B \cap \left( B_{\Gamma r}(x) \setminus B_{\gamma r}(x) \right)} \frac{dy}{|x - y|^{1 + \alpha}} \right) dx 
	&\ge \int_{A \cap Q^{(a)}} \left( \int_{B \cap Q^{(b)}} \frac{dy}{|x - y|^{1 + \alpha}} \right) dx  \\
	&\ge \dfrac{|A \cap Q^{(a)}| |B \cap Q^{(b)}| }{ (4r)^{1+ \alpha}} \ge \dfrac{1}{9 \cdot 4^{1+\alpha}} \, r^{1 - \alpha},
	\end{split}
\end{equation}
which gives~\eqref{Eq:mainintest}.

\smallskip

\noindent
\textbf{- Case 3.2. Dimension~$n\geq 2$.}
This case requires a bit more of work.
The idea is to mimic the one-dimensional argument in columns of cubes (along a certain direction) where we have cubes of~$\K_A$ and~$\K_B$ but not in~$\K_E$.
For a given column satisfying the previous property, the same argument as in dimension one gives a lower bound of the form~$c \, r^{n-\alpha}$ for some~$c>0$, which leads to the desired estimate if we show that the number of such columns is comparable to~$r^{1-n}$.
Note that in this case we cannot reduce ourselves to the case~$\K_E=\varnothing$---this would require~$\varepsilon < r^n/3$, whereas~$\varepsilon$ will need to be comparable to~$r$. Therefore, we will also need to estimate the number of columns with cubes in~$\K_E$.
The details are as follows.

First of all, thanks to~\eqref{Eq:QBinA/QAinBsmall},~\eqref{ABaresigmafat}, and~\eqref{QEest}, we see that
$$
|Q_A| \ge |Q_A \cap A| = |A| - |Q_E \cap A| - |Q_B \cap A| \ge |A| - |Q_E| - r \ge \sigma - 3 \varepsilon - r \ge \frac{\sigma}{2}
$$
and
$$
|Q_B| \ge |Q_B \cap B| = |B| - |Q_E \cap B| - |Q_A \cap B| \ge |B| - |Q_E| - r \ge \sigma - 3 \varepsilon - r \ge \frac{\sigma}{2},
$$
since~$r \le \frac{\sigma}{4}$ and~$\varepsilon \le \frac{\sigma}{12}$, provided we take~$\gamma \leq \frac{1}{3}$. In particular,
\begin{equation} \label{QAbounds}
	|Q_A| \in \left[ \frac{\sigma}{2}, 1 - \frac{\sigma}{2} \right].
\end{equation}

Now, in light of~\cite{SV14}*{Lemma~4.1}, the projection of~$Q_A$ along at least one of the Cartesian axes has~$(n - 1)$-dimensional measure greater or equal to~$|Q_A|^{\frac{n - 1}{n}}$. 
Then, up to relabeling the axis, we may assume that
\begin{equation} \label{QAfatproj}
\Haus^{n - 1} \big( {\Pi_{e_n}(Q_A)} \big) \ge |Q_A|^{\frac{n - 1}{n}},
\end{equation}
where~$\Pi_{e_n}$ denotes the projection along~$e_n$.
We then consider the following columns of cubes.
First, we set
\begin{equation}
	\C_A  \coloneqq \left\{ Q_1 \cap \bigcup_{j \in \Z} \left( Q + r j e_n \right) : Q \in \K_A \right\},
\end{equation}
which are the columns of cubes which contain cubes in~$\K_A$.
We denote by~$M_A$ the cardinality of~$\C_A$.
Then, among the columns of~$\C_A$ we distinguish three families of columns:
\begin{equation}
\begin{split}
\C_E & \coloneqq \left\{ C \in \C_A : C = Q_1 \cap \bigcup_{j \in \Z} \left( Q + r j e_n \right) \mbox{ for some } Q \in \K_E \right\}, \\
\C_B & \coloneqq \left\{ C \in \C_A \setminus \C_E : C = Q_1 \cap \bigcup_{j \in \Z} \left( Q + r j e_n \right) \mbox{ for some } Q \in \K_B \right\},
\end{split}
\end{equation}
and~$\C_A \setminus \left( \C_E \cup \C_B \right)$---the columns which are only constituted by cubes of~$\K_A$.
We denote with~$m_E$,~$m_B$, and~$m_A$ respectively the cardinalities of these three families.
Our goal is to get a good lower bound for~$m_B$, the number of columns which have at least one cube in~$\K_A$ and one in~$\K_B$ but none in~$\K_E$.
Clearly, it holds
$$
m_B = M_A  - m_A- m_E.
$$
We thus proceed to estimate separately these last three quantities.

First, thanks to~\eqref{QAbounds},
\begin{equation} \label{MAbound}
M_A \, r^{n-1}= \Haus^{n - 1} \big( {\Pi_{e_n}(Q_A)} \big) \ge |Q_A|^{\frac{n - 1}{n}} .
\end{equation}
Second, since~$m_A$ is the number of columns with only cubes in~$\K_A$, we deduce that
\begin{equation} \label{mAbound}
 m_A \, r^{n - 1} \le |Q_A|.
\end{equation}
Last, if~$C \in \C_E$, then there exists~$Q' \in \K_E$ such that~$Q' \subset C$ and thus
$$
|C \cap E| \ge |Q' \cap E| \ge \frac{1}{3} \, |Q'| = \frac{r^n}{3}.
$$
Consequently,
$$
\varepsilon = |E| \ge \left| E \cap \bigcup_{C \in \C_E} C \right| = \sum_{C \in \C_E} |C \cap E| \ge \frac{m_E \, r^n}{3},
$$
which, recalling that~$\varepsilon \le \gamma r$, yields the bound
\begin{equation} \label{mEbound}
m_E \le \frac{3 \gamma}{r^{n - 1}}.
\end{equation}

Thanks to~\eqref{MAbound},~\eqref{mAbound},~\eqref{mEbound}, and~\eqref{QAbounds}, we estimate
\begin{equation} \label{MBbound}
\begin{split}
m_B & = M_A - m_A - m_E \ge r^{1 - n} \left\{ |Q_A|^{\frac{n - 1}{n}} - |Q_A| - 3 \gamma \right\} \\
& = r^{1 - n}  \left\{ |Q_A|^{\frac{n - 1}{n}} \left( 1 - |Q_A|^{\frac{1}{n}} \right) - 3 \gamma \right\}  \\ 
& \ge r^{1 - n} \left\{ \left( \frac{\sigma}{2} \right)^{\! \frac{n - 1}{n}} \left( 1 - \left( 1 - \frac{\sigma}{2} \right)^{\! \frac{1}{n}} \right)  - 3 \gamma \right\} \\
& \ge c_2 \, \sigma^{\frac{2 n - 1}{n}} r^{1 - n},
\end{split}
\end{equation}
for some~$c_2 > 0$ depending only on~$n$, provided~$\gamma \le c_2 \, \sigma^{\frac{2 n - 1}{n}}$.

Let~$C \in \C_B$. Then,~$C \subset Q_A \cup Q_B$,~$C \cap Q_A \ne \varnothing$, and~$C \cap Q_B \ne \varnothing$. 
In light of this, it is easy to see that there exist two cubes~$Q^{(a)} \in \K_A$ and~$Q^{(b)} \in \K_B$ with~$Q^{(a)} \cup Q^{(b)} \subset C$ and such that~$\dist \left( Q^{(a)}, Q^{(b)} \right) \in \left\{ r, 2 r \right\}$. As a result,
$$
\int_{A \cap Q^{(a)}} \int_{B \cap Q^{(b)}} \frac{dx dy}{|x - y|^{n + \alpha}} \ge \frac{|A \cap Q^{(a)}| |B \cap Q^{(b)}|}{\big( {3 + \sqrt{n}} \big)^{n + \alpha} r^{n + \alpha}} \ge \frac{|Q^{(a)}| |Q^{(b)}|}{9 \big( {3 + \sqrt{n}} \big)^{n + \alpha} r^{n + \alpha}} \eqqcolon c_3 \, r^{n - \alpha}.
$$
Also, for every~$x \in Q^{(a)}$ and~$y \in Q^{(b)}$ we have that~$r < |x - y| < (3 + \sqrt{n}) r \le \Gamma r$, provided $\gamma \le \frac{1}{3 + \sqrt{n}}$. Consequently,~$Q^{(b)} \subset B_{\Gamma r}(x) \setminus B_{\gamma r}(x)$ for all~$x \in Q^{(a)}$ and thus
$$
\int_{A \cap C} \left( \int_{B \cap \left( B_{\Gamma r}(x) \setminus B_{\gamma r}(x) \right)} \frac{dy}{|x - y|^{n + \alpha}} \right) dx \ge \int_{A \cap Q^{(a)}} \left( \int_{B \cap Q^{(b)}} \frac{dy}{|x - y|^{n + \alpha}} \right) dx \ge c_3 \, r^{n - \alpha}.
$$
By adding these inequalities up as~$C \in \C_B$ and recalling~\eqref{MBbound} we are easily led to~\eqref{Eq:mainintest}.

\medskip

\noindent
\textbf{Step 2. Conclusion.} To finish the proof, let~$c_\star$ be the constant, depending on~$n$ and~$\alpha$, found in~\eqref{Eq:mainintest} and set
$$
\gamma \coloneqq \left\lceil \frac{1}{c_\star \sigma^{\frac{2n - 1}{n}}} \right\rceil^{-1} \le c_\star \sigma^{\frac{2n - 1}{n}}.
$$
Let~$k_0$ be the largest integer for which~$\gamma^{2 k_0 + 1} \ge \varepsilon$.
Notice that if~$\varepsilon \leq \big( {\frac{c_\star}{2} \, \sigma^{\frac{2n - 1}{n}}} \big)^3$, then~$\varepsilon\leq \gamma^3$.
Indeed, if $N\geq 1$ is the unique natural number such that~$\gamma = \frac{1}{N+1} \leq c_\star  \sigma^{\frac{2n - 1}{n}} < \frac{1}{N}$, we have~$\varepsilon \leq \left(\frac{1}{2N}\right)^3 \leq (N+1)^{-3} = \gamma^3$.
As a consequence, we have  $k_0\geq 1$.

Now, for every~$k \in \{ 1, \ldots, k_0 \}$, we set~$r_k \coloneqq \gamma^{2 k}$. 
Note that~$\frac{1}{r_k} \in \N$ and~$r_k \in \left[ \varepsilon/\gamma, \sigma/4 \right]$. 
Hence, we may apply~\eqref{Eq:mainintest} with~$r = r_k$ and deduce that
$$
\iint_{\big\{ {(x, y) \in A \times B \, : \, \gamma^{2 k + 1} \le |x - y| \le \gamma^{2k - 1}} \big\}} \frac{dx dy}{|x - y|^{n + \alpha}} \ge c_\star \, \sigma^{\frac{2n - 1}{n}} \gamma^{(1 - \alpha) 2k} \quad \mbox{for all } k \in \{ 1, \ldots, k_0 \}.
$$
Summing up over~$k$, we obtain
\begin{equation}
\int_A \int_B \frac{dx dy}{|x - y|^{n + \alpha}}  \ge \sum_{k = 1}^{k_0} \iint_{\big\{ {(x, y) \in A \times B \, : \, \gamma^{2 k + 1} \le |x - y| \le \gamma^{2k - 1}} \big\}} \frac{dx dy}{|x - y|^{n + \alpha}} 
 \ge c_\star \sigma^{\frac{2n - 1}{n}} \sum_{k = 1}^{k_0} \gamma^{(1 - \alpha) 2k}.
\end{equation}
On the one hand, for~$\alpha > 1$ we use that~$\gamma^{2k_0 + 3}<\varepsilon$ and~$\gamma \leq c_\star \sigma^{\frac{2n - 1}{n}}<2\gamma$ to conclude that
\begin{equation}
	\int_A \int_B \frac{dx dy}{|x - y|^{n + \alpha}} \ge c_\star \sigma^{\frac{2n - 1}{n}} \gamma^{(1 - \alpha) 2k_0} \ge c_\circ \frac{\sigma^{\frac{(2n - 1)(3 \alpha - 2)}{n}}}{\varepsilon^{\alpha - 1}} = c_\circ \frac{\sigma^{\frac{(2n - 1)(3 \alpha - 2)}{n}}}{|E|^{\alpha - 1}},
\end{equation}
% \begin{cases}
%\gamma^{(1 - \alpha) 2k_0} & \quad \mbox{if } \alpha > 1, \\
%k_0 & \quad \mbox{if } \alpha = 1.
%\end{cases}
for some~$c_\circ > 0$ depending only on~$n$ and~$\alpha$.
On the other hand, for~$\alpha = 1$ we have
\begin{equation}
	\int_A \int_B \frac{dx dy}{|x - y|^{n + 1}} \ge c_\star \sigma^{\frac{2n - 1}{n}} k_0.
\end{equation}
Since~$\gamma^{2k_0 + 3}<\varepsilon$, using that~$\gamma \leq c_\star \sigma^{\frac{2n - 1}{n}}\leq \sigma^{\frac{2n - 1}{n}}$ and~$\varepsilon \le \frac{1}{4}$ it follows that 
\begin{equation}
	5 k_0 \geq 2k_0 + 3 \geq \dfrac{|{\log \varepsilon}|}{|{\log \gamma}|} \geq \dfrac{n}{2n -1} \dfrac{|{\log \varepsilon}|}{|{\log \sigma}|} \ge \dfrac{\left| \log \! \left(\frac{\varepsilon}{e} \right) \right|}{8 |{\log \sigma}|}.
\end{equation} 
Thus,
\begin{equation}
	\int_A \int_B \frac{dx dy}{|x - y|^{n + 1}} \ge c_\circ  \dfrac{\sigma^{\frac{2n - 1}{n}}}{|{\log \sigma}|} \left| \log \! \left(\frac{\varepsilon}{e} \right) \right|,
\end{equation}
for some~$c_\circ > 0$ depending only on~$n$. This leads to estimate~\eqref{Eq:interacest} also when~$\alpha = 1$ and concludes the proof of Lemma~\ref{Lemma:SVest}.

\section{Proof of Theorem~\ref{Th:DeGiorgiIsop}}	
\label{Sec:MainResult}

\noindent
First of all, by scaling we may reduce ourselves to the case~$h = 0$ and~$k = 1$. Secondly, it suffices to prove the estimate for interfaces of small densities, i.e., for which~$\frac{|B_1 \cap \{ 0 < u < 1 \}|}{|B_1|} \le \delta$, with~$\delta \in \left( 0, \frac{1}{10} \right]$ fixed. Indeed, if instead~$\frac{|B_1 \cap \{ 0 < u < 1 \}|}{|B_1|} > \delta$, then it trivially holds
\begin{equation} \label{trivialtechest}
\begin{split}
[u]_{W^{s, p}(B_1)}^p & \ge \int_{B_1 \cap \{ u \le 0 \}} \int_{B_1 \cap \{ u \ge 1 \}} \frac{dx dy}{|x - y|^{n + sp}} \ge \frac{|B_1 \cap \{ u \le 0 \}| |B_1 \cap \{ u \ge 1 \}|}{2^{n + sp}} \\
& \ge \frac{\Psi_{s p}(\delta)}{2^{n + sp}} \, \dfrac{|B_1 \cap \{ u \le 0 \}| |B_1 \cap \{ u \ge 1 \}|}{\Psi_{s p} \! \left( \frac{|B_1 \cap \{ 0 < u < 1 \}|}{|B_1|} \right)},
\end{split}
\end{equation}
where we took advantage of the monotonicity of the function~$\Psi_\alpha$---see~\eqref{Eq:Psialphadef} for its definition.

To deal with interfaces of small densities, we transfer the problem from the unit ball to the unit cube and take advantage of the inequality of Lemma~\ref{Lemma:SVest}. To this aim, we consider a bi-Lipschitz diffeomorphism~$\Phi: \overline{Q_1} \to \overline{B_1}$---such as the composition of the map~$x \mapsto \big( {|x|_\infty / |x|} \big) x$, where~$| \cdot |_\infty$ is the~$\infty$-norm in~$\R^n$, with an affine transformation mapping~$\overline{Q_1}$ onto~$[-1, 1]^n$. Let~$v \coloneqq u \circ \Phi$ and note that it suffices to establish our estimate with~$B_1$ replaced by~$Q_1$ and~$u$ by~$v$, if~$\delta$ is small enough compared to the Lipschitz norm of~$\Phi$ (this is only really needed when~$s p = 1$). Indeed,
\begin{equation}
\begin{split}
[v]_{W^{s, p}(Q_1)} & \le M [u]_{W^{s, p}(B_1)}, \vphantom{\frac{|B_1 \cap \{ 0 < u < 1 \}|}{|B_1|}} \\
|B_1 \cap \{ u \le 0 \}| |B_1 \cap \{ u \ge 1 \}| & \le M |Q_1 \cap \{ v \le 0 \}| |Q_1 \cap \{ v \ge 1 \}|,  \\
|Q_1 \cap \{ 0 < v < 1 \}| & \le M \frac{|B_1 \cap \{ 0 < u < 1 \}|}{|B_1|},
\end{split}
\end{equation}
for some~$M > 1$ depending only on~$n$,~$s$,~$p$, and~$\Phi$. From this, our claim immediately follows when~$s p > 1$. It follows as well when~$s p = 1$ provided we take~$\delta \le \frac{1}{2 M}$, since then~$M \frac{|B_1 \cap \{ 0 < u < 1 \}|}{|B_1|} \le \frac{1}{2}$ and therefore
\begin{equation}
\begin{split}
\Psi_1 \Big( {|Q_1 \cap \{ 0 < v < 1 \}|} \Big) & \le \Psi_1 \! \left( M \frac{|B_1 \cap \{ 0 < u < 1 \}|}{|B_1|} \right) \\
& \le \frac{1}{1 + \frac{\log M}{\log \left( \frac{\delta}{e} \right)}} \, \Psi_1 \! \left( \frac{|B_1 \cap \{ 0 < u < 1 \}|}{|B_1|} \right),
\end{split}
\end{equation}
as can be easily deduced from the fact that the function~$t \mapsto \frac{\Psi_1(Mt)}{\Psi_1(t)}$ is increasing in~$(0, \delta]$.

We are thus left to establish our estimate for~$v$ in the fractional Sobolev space~$W^{s, p}$ over the cube~$Q_1$. Set
$$
A \coloneqq Q_1 \cap \{ v \le 0 \}, \quad B \coloneqq Q_1 \cap \{ v \ge 1\}, \quad E \coloneqq Q_1 \setminus \left( A \cup B \right) = Q_1 \cap \{ 0 < v < 1 \}.
$$
When either~$A$ or~$B$ has measure zero, the inequality is trivially verified. Hence, we may assume without loss of generality that~$\sigma \coloneqq \min \{ |A|, |B| \} > 0$. Recalling~\eqref{Eq:seminormcontrolsinteraction}, we see that the nonlocal interaction~$I(A, B)$ is bounded above by the~$p$-th power of the~$W^{s, p}$-seminorm of~$v$ and is thus finite. Hence, in view of~\cite{B02}*{Corollary~2}, we infer that~$|E| > 0$.
Let~$c_\sharp \in (0, 1)$ and~$q \ge 1$ be the constants found in Lemma~\ref{Lemma:SVest}, in correspondence to~$\alpha = sp$. In view of this result, if~$|E| \le c_\sharp \, \sigma^q$, then
$$
\int_A \int_B \frac{dx dy}{|x - y|^{n + sp}} \ge c_\sharp \, \frac{\sigma^q}{\Psi_\alpha \big( {|E|} \big)}.
$$
If, on the other hand~$|E| > c_\sharp \, \sigma^q$, then much like in~\eqref{trivialtechest} we simply estimate
$$
\int_A \int_B \frac{dx dy}{|x - y|^{n + s p}} \ge \frac{|A| |B|}{n^{\frac{n + sp}{2}}} \ge \frac{\sigma^2}{n^{\frac{n + sp}{2}}} \begin{dcases}
c_\sharp^{sp - 1} \sigma^{q (sp - 1)} |E|^{1 - sp} & \quad \mbox{if } sp > 1, \\
\frac{\left| \log \left( \frac{|E|}{e} \right) \right|}{\left| \log \left( \frac{c_\sharp \sigma^q}{e} \right) \right|} & \quad \mbox{if } sp = 1.
\end{dcases}
$$
From the last two inequalities,~\eqref{Eq:seminormcontrolsinteraction}, and the fact that~$\sigma = \min \big\{ {|A|, |B|} \big\} \ge |A| |B|$, our claim easily follows, with~$\beta$ sufficiently large in dependence of~$n$,~$s$, and~$p$.

\section{Proof of Theorem~\ref{Th:DeGHolder}}	
\label{Sec:Application}

\noindent
The local boundedness of~$u$ is already been established in~\cite{C19}*{Theorem~2.2}. To obtain its H\"older continuity, we follow the strategy of~\cites{Cozzi-DeGiorgi,C19}. In particular, recalling the discussions succeeding the statement of~\cite{C19}*{Lemma~2.5}, it suffices to show that functions belonging to weak De Giorgi classes satisfy a suitable growth lemma. More precisely, Theorem~\ref{Th:DeGHolder} boils down to proving the following result.

\begin{lemma}
	Let~$n \in \N$,~$s\in (0,1)$, and~$p \in (1, +\infty)$ be such that~$s p \ge 1$. Let~$H, \lambda \ge 0$. For every~$\tau \in (0, 1)$, there exists~$\delta \in (0, 1/8]$, depending only on~$n$,~$s$,~$p$,~$H$,~$\lambda$, and~$\tau$, such that if, for~$d \ge 0$, a function~$u\in \wDG_{-}^{s, p}(B_4; d, H, \lambda)$ satisfies
	\begin{gather}
		\label{Eq:u>0}
		u \ge 0 \quad \mbox{in } B_4, \\
		\label{Eq:Densityu>1} \dfrac{|B_2 \cap \{u\geq 1\}|}{|B_2|} \ge \dfrac{1}{2}, \\
		\label{Eq:SmallnessHypothesis}
		d + \Tail_{s,p} \! \big( {u_-; B_4} \big) \leq \delta,
	\end{gather}
	then
	\begin{equation}
		\label{Eq:Densityu<2delta}
		\dfrac{|B_2 \cap \{u< 2\delta\}|}{|B_2|} \le \tau.
	\end{equation}
\end{lemma}

\begin{proof}
	First of all, after a simple rescaling we see that every~$v \in W^{s, p}(B_2)$ satisfies the isoperimetric inequality of Theorem~\ref{Th:DeGiorgiIsop} in the ball~$B_2$, that is
	\begin{equation}
		\label{Eq:isopineinB2}
		\left[ { \frac{|B_2 \cap \{ v \le h \}|}{|B_2|} \frac{|B_2 \cap \{ v \ge k \}|}{|B_2|}} \right]^{\gamma} \le \frac{C_\circ}{(k - h)^p} \, [v]_{W^{s, p}(B_2)}^p \Psi_{sp} \! \left( \frac{|B_2 \cap \{ h < v < k \}|}{|B_2|} \right),
	\end{equation}
	for every~$0 \le h < k$ and for some constants~$C_\circ, \gamma > 0$ depending only on~$n$,~$s$, and~$p$. The idea is now to combine the Caccioppoli inequality defining weak De Giorgi classes with~\eqref{Eq:isopineinB2}, applied at suitable truncations of~$u$ and with different pairs of levels~$(h, k)$, to eventually obtain that the sublevel set~$B_2 \cap \{ u < 2 \delta \}$ becomes smaller and smaller in measure, as~$\delta$ approaches zero. It is worth noting that the only properties of the function~$\Psi_{sp}$ that we will use are its monotonicity and the fact that~$\lim_{t \downarrow 0}\Psi_{sp}(t) = 0$.

	Let~$\tau \in (0, 1)$ be fixed and~$\delta \in (0, 1)$ to be chosen later. We claim that
	\begin{equation}
		\label{Eq:GrowthLemmaProofSeminormEstimate}
		\seminorm{(u-k)_-}^p_{W^{s,p}(B_2)} \leq C_1 k^p \quad \mbox{for every } k \ge \delta,
	\end{equation}
	for some constant~$C_1 \ge 1$ depending only on~$n$,~$s$,~$p$,~$H$, and~$\lambda$. To accomplish this, we apply the Caccioppoli inequality~\eqref{Eq:weakDeGclass} for the class~$\wDG_{-}^{s, p}$ with~$k\geq0$,~$x_0=0$,~$r = 2$, and~$R= 3$. We obtain that
	\begin{equation}
		\label{Eq:CacciopApplied}
		\begin{split}
			\seminorm{(u-k)_-}^p_{W^{s,p}(B_2)} & \leq C \, \bigg\{ d^p |B_3 \cap \{u<k\}|
			+ 
			\norm{(u-k)_-}^p_{L^p(B_3)} \\
			& \quad
			+ \norm{(u-k)_-}_{L^1(B_3)} \Tail_{s,p} \! \big( {(u-k)_-; B_2} \big)^{p-1} \bigg\},
		\end{split}
	\end{equation}
	where~$C$ denotes here and in the following computations a positive constant, depending only on~$n$,~$s$,~$p$,~$H$, and~$\lambda$, whose value may change from line to line. By recalling the non-negativity assumption~\eqref{Eq:u>0}, we get the following estimates.
	On the one hand, for every~$q\geq 1$ we have
	\begin{equation}
			\norm{(u-k)_-}^q_{L^q(B_3)}  = \int_{B_3\cap \{u<k\}} \big( {k-u(x)} \big)^q \, dx \leq k^q |B_3\cap \{u<k\}| \leq C k^q.
	\end{equation}
	On the other hand,
	\begin{equation}
		\begin{split}
			\Tail_{s,p} \!  \big( {(u-k)_-; B_2} \big)^{p-1}  & 
			= \int_{\R^n \setminus B_2} \dfrac{\big( {u(x)-k} \big)_-^{p-1}}{|x|^{n + sp}} \, dx 
			= \int_{\{ u < k \} \setminus B_2} \dfrac{\big( {k-u(x)} \big)^{p-1}}{|x|^{n + sp}} \, dx\\
			& \leq 2^{p-1} \left(
			\int_{\R^n \setminus B_2} \dfrac{k^{p-1}}{|x|^{n + sp}} \, dx 
			+ \int_{\{u < k\}\setminus B_2} \dfrac{|u(x)|^{p-1}}{|x|^{n + sp}} \, dx
			\right) \\
			& \leq C \left( k^{p-1}
			+ \int_{\{0 \leq u < k \} \setminus B_2} \dfrac{u(x)^{p-1}}{|x|^{n + sp}} \, dx
			+ \int_{\{u<0\} \setminus B_2} \! \dfrac{u_-(x)^{p-1}}{|x|^{n + sp}} \, dx
			\right) \\
			& \leq C \, \Big( {k^{p-1}
			+ \Tail_{s,p} \!  \big( {u_-; B_4} \big)^{p-1}} \Big).
		\end{split}
	\end{equation}
	By plugging these two estimates into~\eqref{Eq:CacciopApplied} (the first one with both~$q= p$ and~$q=1$), we arrive at
	\begin{equation}
		\seminorm{(u-k)_-}^p_{W^{s,p}(B_2)}  \leq C \, \Big( {d^p + k^p + k \, \Tail_{s,p} \!  \big( {u_-; B_4} \big)^{p-1}} \Big).
	\end{equation}
	By applying the smallness hypothesis~\eqref{Eq:SmallnessHypothesis} on~$d$ and on the tail of~$u_-$ outside~$B_4$, we are immediately led to claim~\eqref{Eq:GrowthLemmaProofSeminormEstimate}.
	
	At this point, we combine estimate~\eqref{Eq:GrowthLemmaProofSeminormEstimate} with the isoperimetric type inequality~\eqref{Eq:isopineinB2}.
	Write~$k_j \coloneqq 2^{-j}$ for any~$j \in \Z$. Let then~$M$ be the unique integer such that~$k_{M + 3} \leq \delta < k_{M + 2}$. 
	Recall that~$\delta$ is still to be chosen and observe that if~$\delta \in \left(0, \frac{1}{8} \right)$, then~$M\geq 1$.
	We plan to use the inequality~\eqref{Eq:isopineinB2} to estimate the measure of certain level sets of~$u$, in particular, those of the form~$B_2 \cap \{ k_{j + 1} < u < k_j \}$.
	To do it, we apply~\eqref{Eq:isopineinB2} with~$v = (u-k_{j-1})_-$,~$h = k_{j-1} - k_j = 2^{-j}$, and~$k = k_{j-1} - k_{j+1} = 3\cdot 2^{-j-1}$, for all~$j = 1, \ldots, M$.
	We get
	\begin{equation}
	\begin{split}
		& \left[ { \frac{\left| B_2 \cap \big\{ {(u-k_{j-1})_- \leq 2^{-j}} \big\} \right|}{|B_2|} \frac{\left| B_2 \cap \big\{ {(u-k_{j-1})_- \geq 3\cdot 2^{-j-1}} \big\} \right|}{|B_2|}} \right]^{\gamma} \\
		& \hspace{30pt} \leq C_\circ \dfrac{\seminorm{(u- k_{j-1})_-}^p_{W^{s,p}(B_2)}}{(2^{-j-1})^p} 
		\, \Psi_{s p} \! \left(\frac{\left| B_2 \cap \big\{ {2^{-j} < (u-k_{j-1})_- < 3\cdot 2^{-j-1}} \big\} \right|}{|B_2|}\right),
	\end{split}
	\end{equation}
	which is equivalent to
	\begin{equation}
		\begin{split}
			& \left[ { \frac{\left|B_2 \cap \big\{ {u \geq 2^{-j}} \big\} \right|}{|B_2|} \frac{\left|B_2 \cap \big\{ {u \leq 2^{-j-1}} \big\} \right|}{|B_2|}} \right]^\gamma
			\\
			& \hspace{60pt} \leq C_2 \dfrac{\seminorm{(u- k_{j-1})_-}^p_{W^{s,p}(B_2)}}{
				k_{j-1}^p} \,
			\Psi_{s p} \! \left(\dfrac{\left|B_2 \cap \big\{ {2^{-j-1} < u < 2^{-j}} \big\} \right|}{|B_2|}\right),
		\end{split}
	\end{equation}
	with~$C_2 \coloneqq 4^p C_\circ$. Since~$k_{j-1} > \delta$ for every~$j = 1, \ldots, M$, we can apply estimate~\eqref{Eq:GrowthLemmaProofSeminormEstimate} with~$k=k_{j-1}$. This readily gives
	\begin{equation}
		\dfrac{\seminorm{(u- k_{j-1})_-}^p_{W^{s,p}(B_2)}}{
			k_{j-1}^p} \leq C_1 \quad \text{for every } j = 1, \ldots, M.
	\end{equation}
	As a consequence, setting~$C_3 \coloneqq C_1 C_2$ we obtain
	\begin{equation} \label{Eq:uleveljtechest}
		\frac{1}{C_3} \left[ {\frac{\left|B_2 \cap \big\{ {u \geq 2^{-j}} \big\} \right|}{|B_2|} \frac{\left|B_2 \cap \big\{ {u \leq 2^{-j-1}} \big\} \right|}{|B_2|}} \right]^\gamma
		\leq \Psi_{s p} \! \left(\dfrac{\left|B_2 \cap \big\{ {2^{-j-1} < u < 2^{-j}} \big\} \right|}{|B_2|}\right).
	\end{equation}
	Note now that the ``large density'' hypothesis~\eqref{Eq:Densityu>1} yields that
	\begin{equation}
		\dfrac{|B_2 \cap \{u \geq 2^{-j}\}|}{|B_2|} \geq \dfrac{|B_2 \cap \{u \geq 1\}|}{|B_2|} \geq \dfrac{1}{2}.
	\end{equation}
	On the other hand, since~$j \le M$ we have that~$2^{-j-1} \ge 2^{-M- 1} = 2 k_{M + 2} > 2 \delta$ and therefore
	\begin{equation}
		\dfrac{|B_2 \cap \{u \le 2^{-j - 1}\}|}{|B_2|} \ge \dfrac{|B_2 \cap \{u \le 2 \delta \}|}{|B_2|}.
	\end{equation}
	By combining the last two inequalities with~\eqref{Eq:uleveljtechest}, we find that
	\begin{equation}
		\frac{1}{C_4} \, |B_2 \cap \{u \le 2 \delta \}|^\gamma
		\leq \Psi_{s p} \! \left(\dfrac{\left|B_2 \cap \big\{ {2^{-j-1} < u < 2^{-j}} \big\} \right|}{|B_2|}\right) \quad \mbox{for every } j = 1, \ldots, M,
	\end{equation}
	where~$C_4 \coloneqq \big( {2 |B_2|} \big)^\gamma C_3$. Recall that~$\Psi_{s p}$ is strictly increasing and thus invertible. By applying its inverse~$\Psi_{s p}^{-1}$ to both sides of the above inequality and adding up the resulting estimates over~$j$, we obtain
	\begin{equation}
		M \Psi_{s p}^{-1} \! \left( \frac{1}{C_4} \, |B_2 \cap \{u \le 2 \delta \}|^\gamma \right) \le \sum_{j = 1}^{M} \dfrac{\left|B_2 \cap \big\{ {2^{-j-1} < u < 2^{-j}} \big\} \right|}{|B_2|} \le \dfrac{\left|B_2 \cap \{ u < 1 \} \right|}{|B_2|} \le 1.
	\end{equation}
	By this and the fact that~$M \ge - \log_2 \delta - 3 \ge \frac{1}{2} |{\log \delta}|$, provided~$\delta \in \left( 0, \frac{1}{64} \right]$, we conclude~that
	\begin{equation}
		|B_2 \cap \{u \le 2 \delta \}|^\gamma \le C_4 \, \Psi_{s p} \! \left( \frac{1}{M} \right) \le C_4 \, \Psi_{s p} \! \left( \frac{2}{|{\log \delta}|} \right).
	\end{equation}
	Clearly, this estimate leads to claim~\eqref{Eq:Densityu<2delta}, by taking~$\delta$ sufficiently small, in dependence of~$n$,~$s$,~$p$,~$H$,~$\lambda$, and~$\tau$. The proof is thus complete.
\end{proof}

%%%%%%%%%%%%%%%%%%%%%%%%%%%%%%%%%%%%%%%%%%%%%%%%%%%%%%%%%%%%
%%%%%%%%%%%%%%%%%%%%%%%%%%%%%%%%%%%%%%%%%%%%%%%%%%%%%%%%%%%%

\end{document}